\documentclass{amsart}
\usepackage{graphicx} 
\usepackage{chngcntr}
\usepackage{apptools}
\usepackage{color}
\usepackage{amsthm,amssymb,verbatim}
\usepackage{graphicx}
\usepackage{enumerate}
\newcommand{\Sing}{\operatorname{Sing}}

 \newcommand{\genus}{\operatorname{genus}}

  \newcommand{\length}{\operatorname{length}}
  
  \newcommand{\Ss}{\mathbf{S}}
 
 \newcommand{\RR}{\mathbf{R}}  
 \newcommand{\Reg}{\operatorname{Reg}}
 \newcommand{\Tfat}{T_\textnormal{fat}}
 \newcommand{\Tpos}{T_\textnormal{pos}}
 
    \newcommand{\dist}{\operatorname{dist}}
\newcommand{\Min}{M_\textnormal{in}}
\newcommand{\Mout}
{M_\textnormal{out}}
 
  \newcommand{\sing}{\operatorname{sing}}
 
 \newcommand{\eps}{\epsilon}

\newcommand{\Hh}{\mathcal{H}}
\usepackage{amsthm}

\input epsf
\def\begfig {
\begin{figure}
\small }
\def\endfig {
\normalsize
\end{figure}
}

    \newtheorem{theorem}    {Theorem}   
    \newtheorem{lemma}      [theorem]       {Lemma}
    \newtheorem{corollary}  [theorem]     {Corollary}

    \newtheorem{claim}{Claim}
    \newtheorem*{theorem*}{Theorem}
    \theoremstyle{definition}
    \newtheorem{definition}  [theorem] {Definition}
     
    \theoremstyle{definition}
    \newtheorem{remark}   [theorem]       {Remark}
    
\usepackage{enumerate}
\usepackage{hyperref}
\usepackage[alphabetic, msc-links, backrefs]{amsrefs}
\title{The 
Genus-Decreasing Property of Mean Curvature Flow, I}

\author{Brian White}
\address{\newline Department of Mathematics \newline
 Stanford University \newline 
  Stanford, CA 94305, USA\newline
{\sl E-mail address:} {\bf bcwhite@stanford.edu}
}

\date{January 20, 2026}

\begin{document}

\begin{abstract}
This paper proves that,
in mean curvature flow of a compact surface in 
a complete $3$-manifold
  with Ricci curvature bounded
  below, the genus of the regular set
is a decreasing function
of time
as long as the only singularities are given by shrinking sphere and shrinking cylinder tangent flows. 
The paper also proves some local versions of that fact.

\end{abstract}

\maketitle

\section{Introduction}

This paper proves
theorems that imply the following:

\begin{theorem}
\label{intro-theorem}
Suppose $M$ is a compact, smoothly embedded surface in a complete, smooth Riemannian $3$-manifold
with Ricci curvature bounded below.  Let $M(t)$
be the result of flowing
$M$ for time $t$ by level set flow.  Let $\Tpos=\Tpos(M)$ be the largest time $\le \infty$
such that
\begin{equation*}
  t\in[0,\Tpos)\mapsto 
    \mu(t):=\Hh^2\llcorner M(t)
\tag{*}
\end{equation*}
is a standard Brakke flow (see Definition~\ref{standard-definition}), 
and such that all shrinkers to the flow at  times $<\Tpos$
 have multiplicity~$1$ 
 and are planes, spheres, or cylinders.
Then
\begin{equation}
\label{intro-g}
  g(t):= \genus(\Reg M(t))
\tag{*}
\end{equation}
is a decreasing
function of $t\in [0,\Tpos)$,
and of $t\in [0,\Tpos]$ if $\Tpos<\infty$.
\end{theorem}

(By Theorem~\ref{main-theorem}, $g(t)$ is 
decreasing for $t\in [0,\Tpos)$.
By lower-semi\-continuity (Lemma~\ref{lsc-lemma}), 
$g(t)$ is also
decreasing for $t\in [0,\Tpos]$.)

The ``pos'' in $\Tpos$  is short for ``positive genus'':
one can show that if $\Tpos<\infty$, then there is a shrinker at time $\Tpos$ that is a smooth, multiplicity $1$ surface of genus $\ge 1$.
See Theorem~\ref{Tpos-shrinker-theorem} in the appendix.  (In some other works,
$\Tpos$ is denoted by $T_\textnormal{gen}$.)

We also prove localized versions of
Theorem~\ref{intro-theorem}.
See Theorems~\ref{local-open} and~\ref{local-curves}, and Remark~\ref{local-open-remark}.
Theorem~\ref{local-curves}
plays an essential role
in~\cite{hmw-shrinkers}.

In a subsequent
paper~\cite{genus-II}, we will show that $g(t)$
is decreasing for
$t\in [0,\Tfat]$, where
$\Tfat=\Tfat(M)$ is the fattening time, i.e., the infimum of
times $t$ at which $M(t)$ has nonempty interior.

Now suppose that $M$ is two-sided, and let $\Min(\cdot)$ and $\Mout(\cdot)$ be the innermost and outermost mean curvature
flows starting with $M$.
Thus $\Min(t)=\Mout(t)=M(t)$
for all $t\le \Tfat$,
but $\Min(t)\ne\Mout(t)$
for all $t>\Tfat$ sufficiently close to~$t$.  
In~\cite{genus-II}, we will give
an example to show that 
\[
   \genus(\Reg \Min(t))
\]
need {\bf not} be a decreasing function of~$t$. On the 
other hand, its
restriction to the set
of regular times is
decreasing.
(The same statements are 
true for $\Mout(\cdot)$.)

The first paper to prove a genus-decreasing property of mean curvature flow
was~\cite{white-95}.
That paper showed that
if $M$ is a smooth, compact surface in $\RR^3$, and if $t_1<t_2$
are times at which
 $M(t_i)$ is a smooth embedded surface, then
\[
  \genus(M(t_1))
  \ge
  \genus(M(t_2)).
\]
More generally,
if one assumes completeness and a lower bound on Ricci curvature, then that proof works in ambient
manifolds diffeomorphic to $\RR^3$ or to $\Ss^3$.
(With a little modification, the proof also works for manifolds diffeomorphic to any open subset of $\Ss^3$).
The proof is rather easy: it uses very little other than the avoidance principle.
But the proof breaks down
for general ambient $3$-manifolds.
In~\cite{ccms-generic}*{Appendix~G}, Chodosh, Choi, Mantoulidis,
and Schulze gave localized versions of some of the 
theorems in~\cite{white-95}, but not of
the theorem about genus reduction.

See~\cite{white-topological} for some 
higher dimensional results about change of topology
under mean curvature flow.

In~\cite{bk}, Bamler and Kleiner sketched
a proof that
the function $g(\cdot)$
in Theorem~\ref{intro-theorem} is 
decreasing when restricted to the set of regular times.
See~\cite{bk}*{Lemma~7.8}. (This monotonicity plays an important role in further arguments in~\cite{bk}.) 
It may be helpful to some readers to have a more detailed proof, such as the one   provided here. This proof appears to differ from the one outlined in~\cite{bk}, and it also yields the stronger conclusion that $g(\cdot)$ is decreasing for all times $\le \Tpos$, not just for regular times.

In~\cite{bk}, Bamler and Kleiner deduced (from the $t<\Tpos$ case)
that the restriction of $g(\cdot)$
to the set of regular times is decreasing up 
to $\Tfat$.
More generally, they deduce that if $t_1<t_2$ are regular times for $\Min(\cdot)$, then
$\genus(\Min(t_1))
\ge
\genus(\Mout(t_2))$.
(The same is true for 
  $\Mout(\cdot)$.)

The key tool
in  this paper is the following
fact (Theorem~\ref{geodesics-theorem}):
if $T<\Tpos$ is a singular time, if $t_k>T$
are regular times converging to $T$, and
if $\Gamma_k$ is a $1$-cycle that 
is length-minimizing mod $2$ in $M(t_k)$, 
then the $\Gamma_k$ are 
bounded away from $\Sing M(T)$ as $k\to\infty$.
To see how this fact is used to control genus, suppose that each $M(t_k)$ has genus $\ge 1$.  Then $M(t_k)$ contains $1$-cycles $\alpha^k$ and $\beta^k$
whose mod~$2$ intersection number is~$1$. 
We choose the cycles to minimize length in their mod~$2$ homology classes. 
Because the $\alpha^k$ and $\beta^k$ are bounded away from $\Sing M(T)$, they converge smoothly (after passing to a subsequence)
to $1$-cycles $\alpha$ and $\beta$ in $\Reg M(T)$.  Furthermore, the intersection number of $\alpha$ and $\beta$ is $1$, so $\Reg M(T)$ must also have genus $\ge 1$.

The results in this paper could be proved by other methods, such as mean curvature flow with surgery.  The proofs here are perhaps more intuitive
because they only use qualitative statements about mean curvature flow. (Of course, the proofs of those qualitative statements involve hard analysis.)

The paper is organized as follows. 
Section~\ref{prelim-section} summarizes the main facts about mean curvature flow that we need.
 Section~\ref{general-section} proves some very general facts about flows (not necessarily mean curvature flows) and genus.
Section~\ref{spherical-section} proves a simple property of shrinking sphere singularities.
Section~\ref{cylindrical-section} proves some properties of shrinking cylinder singularities.
Section~\ref{geodesics-section} proves the key fact mentioned above about length-minimizing $1$-cycles.
Sections~\ref{genus-section}, 
 \ref{local-section},
 and
 \ref{another-local-section} prove the theorems about decreasing genus.

The author would like to thank
David Hoffman and Francisco Mart\'in for helpful suggestions.

\section{Terminology
and Preliminaries}
\label{prelim-section}

The results in the introduction were stated in terms of level set flow.  In  rest of the paper, most of the theorems are are stated in slightly more generality for integral Brakke flows. Strictly speaking, a Brakke flow is a one-parameter family of Radon measures.  However, for the flows considered in this paper,
all tangent flows have multiplicity~one.
For such flows, the Radon
measures are determined by 
the spacetime support of the flow.
Thus we can, by slight abuse of terminology, act as though the flow were a one-parameter family of sets.

We prove theorems about
mean curvature flows $t\in I\mapsto M(t)$ in a smooth
Riemmannian manifold $N$.
It is very convenient to translate and dilate the surface $M(t)$ as if $N$ were Euclidean space.  In fact, we can do this by assuming that $N$ is isometrically embedded in a Euclidean space.

\begin{definition}
\label{standard-definition}
A {\bf standard Brakke} flow
is an integral Brakke flow that is unit regular and cyclic.
\end{definition}

Unit regularity means that each spacetime point of Gauss density~$1$ has a spacetime neighborhood in which the flow is a classical smooth mean curvature flow with multiplicity~$1$.
See~\cite{white-cyclic}*{Definition~4.1} for the definition of cyclic.  Roughly speaking, the a flow is cyclic if no blowup consists of an odd number (counting multiplicity) of halfplanes that meet along a their common edge.  Such flows occur very naturally.  For example, any flow constructed by elliptic regularization (using integral currents or flat chains mod~$2$) is standard,
and any limit of standard flows is standard. See~\cite{white-mcf-boundary} for basic properties of standard Brakke flows.  

The following two theorems summarize most
of the known facts about mean curvature flow
that will be needed in this paper.

\begin{theorem}
\label{summary-theorem}
Suppose $t\in (a,b)\mapsto M(t)$ is a $2$-dimensional,
standard
Brakke flow in a smooth Riemannian $3$-manifold $N$.   
Suppose $T\in (a,b)$,
$p\in M(T)$, and the tangent
flow to $M(\cdot)$ at $(p,T)$ has multiplicity~one and is a shrinking sphere or shrinking cylinder.
Suppose that $(p_i,t_i)$
is a sequence of regular spacetime points of the flow that converge to $(p,t)$.
Then the norm $h_i$ of the mean curvature vector of $M(t_i)$ at $p_i$ tends to $\infty$, and 
\[
  M_i':= h_i (M(t_i)-p_i)
\]
converges smoothly (after
passing to a subsequence) to a convex limit surface $M'$.

If the tangent flow at $(p,T)$ is a shrinking sphere, then $M'$ is a sphere.

If the tangent flow at $(p,T)$ is a shrinking cylinder with axis $L$, then $M'$ is one 
of the following:
\begin{enumerate}
\item A compact, uniformly convex surface.
\item A bowl soliton.
\item a cylinder.
\end{enumerate}
Furthermore, if $M'$ is a cylinder, then its axis is
parallel to~$L$.
\end{theorem}

Throughout this paper, ``cylinder'' means ``circular cylinder''.

Theorem~\ref{summary-theorem} refers to ``the'' tangent flow, which
is permissible because if one tangent flow is given by a shrinking sphere or shrinking
cylinder (of multiplicity one), then it is the unique tangent flow at that spacetime point.  
See~\cite{cm-unique}. 

\begin{proof}
In the case of a shrinking sphere tangent flow, the assertion is trivially true.
(We will not need that case of the theorem.)

In the case of a shrinking cylinder tangent flow, 
all the assertions, except for last one,  follow from
\cite{chhw}*{Corollary~1.18}.
That corollary
is stated for mean curvature flow in Euclidean space, 
but the proof 
in~\cite{chhw} works in any smooth Riemannian
manifold~\cite{chhw}*{Footnote~(8)}.
The last assertion
 of Theorem~\ref{summary-theorem} follows
from~\cite{cm-new}.
\end{proof}

\begin{theorem}
\label{ricci-theorem}
Suppose $t\in [0,T]\mapsto M(t)$
is a $2$-dimensional standard Brakke flow in a smooth Riemannian $3$-manifold~$N$.
If the support of $M(0)$ is compact, then the union of the supports of the $M(t)$, $t\in [0,T]$, is contained in a compact set.
\end{theorem}

This fact is extremely general: it actually holds for weak set flows
in Riemannian $m$-manifolds.
See Theorems~6 and~34
in~\cite{hw-avoid}.

 \section{Generalities}
\label{general-section} 
If $t\in I\mapsto M(t)$ is any one-parameter family of closed sets of $N$, we let
\[
  \Reg M(t)
\]
be the set of points $p\in M(t)$ with the following property.
There exists an open neighborhood $U$ of $p$ and an a relatively open interval $J\subset I$
such that 
\[
  \tau\in J \mapsto M(\tau)\cap U
\]
is a smooth, one-parameter family of smooth, properly embedded $2$-dimensional, submanifolds of $U$.
We let 
\[
  \Sing M(t):= M(t)\setminus \Reg M(t).
\]

(Warning: whether $p\in \Reg M(t)$ depends not just on $M(t)$, 
but on~$M(\tau)$ for~$\tau$ in some open interval
 containing~$t$.
However, for the flows of interest in this paper, $p\in \Reg M(t)$ if and only there is a neighborhood $U$ of $p$ such that
 $M(t)\cap U$ is a smooth, properly 
 embedded $2$-manifold in $U$.
For instance, this holds for all unit-regular Brakke flows
for which no planes of multiplicity~$>1$ occur as tangent flows.
In particular, it holds
for all flows that are 
almost regular in the terminology of~\cite{bk}.)

We define $\genus(M(t))$ to be $\genus(\Reg M(t))$.

A {\bf regular time} is a time $t\in I$ for which
 $\Sing M(t)$ is empty.

\begin{lemma}\label{lsc-lemma}
$g(t):=\genus(\Reg M(t))$ is a lower-semicontinuous function of $t$.
\end{lemma}

\begin{proof}
Consider a time $T\in I$ and
suppose that $k$ is a nonnegative integer with $g(T)\ge k$.
Then there is $\Sigma\subset \Reg M(T)$ such that $\Sigma$ is a compact region
with smooth boundary and $\genus(\Sigma)\ge k$.

Note there is an $\eps>0$ such that for every $t\in I\cap (T-\eps,T+\eps)$, 
and for every $p\in \Sigma$, there is a unique point $\phi(t,p)$ in $\Reg M(t)$
closest to $p$. Let $\Sigma(t)$ be the image of $\Sigma$ under $\phi(\cdot,t)$.

Then $\Sigma(t)$ and $\Sigma$ are diffeomorphic, so
\[
  k = \genus(\Sigma) = \genus(\Sigma(t)) \le \genus(\Reg(M(t))
\]
for all $t\in I\cap (T-\eps,T+\eps)$.

Thus
\[
  k \le \liminf_{t\to T} g(t).
\]
Since this holds for all integers $k\le g(T)$, it also holds for $k=g(T)$.
\end{proof}

The following lemma is about
arbitrary left-lower-semiconscious functions.
We will apply it to the genus function in Lemma~\ref{lsc-lemma}.

\begin{lemma}\label{general-lemma}
Suppose 
that $g:[a,b]\to [0,\infty]$
is a left-lower-semicontinuous function.
Suppose also that for each $T\in [a,b)$, there exist $t>T$ arbitrarily close to $T$ for which $g(t)\le g(T)$.
Then $g(\cdot)$ is a decreasing
function on $[a,b]$.
\end{lemma}

\begin{proof}
It suffices to prove that $g(b)\le g(a)$.

Let $T$ be the supremum of $t\in [a,b]$ for which $g(t)\le g(a)$.
By left lower semicontinuity,
\begin{equation}
\label{oops}
g(T)\le g(a).
\tag{*}
\end{equation}
If $T<a$, then, by hypothesis, there would exist $t\in (T,b)$
for which
\[
    g(t)\le g(T) \le g(a),
\]
contradicting the choice of $T$.
Thus $T=b$ and therefore
\[
  g(b)=g(T) \le g(a)
\]
by~\eqref{oops}.
\end{proof}

Combining Lemmas~\ref{lsc-lemma} 
and~\ref{general-lemma}
gives the following:

\begin{theorem}\label{general-theorem}
Suppose 
that 
\[
  t\in [a,b]\mapsto M(t)
\]
is a $2$-dimensional,
standard Brakke flow
in a smooth Riemannian manifold~$N$,
and that all tangent flows
to~$M(\cdot)$ are given
by smooth, multiplicity~one shrinkers.  Let
\begin{align*}
&g:I\to [0,\infty], \\
&g(t) = \genus( \Reg M(t)).
\end{align*}
 Suppose $g(\cdot)$ has the 
following property:
\begin{equation}
\label{property}
  \begin{minipage}{0.8\textwidth} 
   If  $t_i>T$
are regular times converging to $T\in I$, 
and if $g(t_i)\ge n$ for all~$i$ (where $n$ is a positive integer), then
$g(T) \ge n$.
  \end{minipage}
  \tag{*}
\end{equation}
Then $g(\cdot)$ is a decreasing function on $[a,b]$. 
\end{theorem}

\begin{proof}
By Lemma~\ref{lsc-lemma}, $g(\cdot)$
is lower semicontinuous.
Now suppose that $T\in [a,b)$.
The hypothesis on shrinkers implies that
almost every time is  regular time
 (cf.~\cite{white-stratification}).
Thus there is a sequence $t_i>T$
of regular times converging to~$T$.
We claim that $g(t_i)\le g(T)$
for all sufficiently large~$i$.
If not, then (after passing to a subsequence) $g(t_i)>g(T)$ for all $i$, which implies that $g(T)<\infty$ and that $g(t_i)\ge g(T)+1$. But then $g(T)\ge g(T)+1$ by Property~\eqref{property}, a contradiction.
Thus $g(t_i)\le g(T)$ for all sufficiently large~$i$.
By Lemma~\ref{general-lemma}, $g(\cdot)$ is a decreasing function.
\end{proof}

\section{Spherical Singularites}
\label{spherical-section}

\begin{lemma}\label{sphere-lemma}
Suppose that $t\in I\mapsto M(t)$
is a $3$-dimensional integral Brakke flow in a smooth Riemannian $3$-manifold.
Suppose that $(0,0)$ is spacetime point
at which the tangent flow
is a multiplicity-one shrinking sphere.
Then there is an $R>0$ and an
$\eps>0$ such that 
for every regular time
 $t\in (0,\eps]$, 
 \[
    M(t)\cap B(0,R) = \emptyset.
 \]
\end{lemma}

The proof is very straightforward, but we include it for completeness.

\begin{proof}
If not, there exist $t_i>0$
converging to $0$ and $p_i\in M(t_i)$ converging to $0$.
Let
\[
  r_i :=|(p_i,t_i)|
  =
  \max \{ |p_i|, \, |t_i|^{1/2}\}.
\]
Dilate the flow $M(\cdot)$
and the points $(p_i,t_i)$ parabolically by 
\[ 
  (x,t) \mapsto (r_i^{-1}x, 
                r_i^{-2}t)
\]
to get a flow $M_i(\cdot)$
and points $(p_i',t_i')$
such that $t_i'\ge 0$, 
\[
  p_i'\in M_i(t_i'),
\]
and $|(p_i',t_i')|=1$.

The $M_i(\cdot)$ converge
to the tangent flow
\begin{equation}\label{sphere-flow}
t\in \RR
\mapsto
M'(t)
=
\begin{cases}
    \partial B(0, 2|t|^{1/2})
    &\text{if $t\le 0$}, \\
    \emptyset
    &\text{if $t>0$}.
\end{cases}
\end{equation}
Let $(p',t')$ be a subsequential
limit of the $(p_i',t_i')$.
Then
\begin{gather*}
  t'\ge 0, \\
  p'\in M'(t'), \\
  |(p',t')| = 1,
\end{gather*}
which contradicts~\eqref{sphere-flow}.
\end{proof}

\section{Cylindrical Singularities}
\label{cylindrical-section}

\begin{theorem}
\label{cylinder-theorem}
Suppose that $t\in I\mapsto M(t)$ is a standard $2$-dimensional Brakke flow in a smooth Riemannian $3$-manifold,
and that $(0,0)$ is a spacetime point at which the tangent flow is a 
 multiplicity-one shrinking cylinder with axis $L$.
Let 
\[
 Q =
 \{(p,t)\ne (0,0): p\in M(t),\, t\ge 0\}.
\]
If $(p_i,t_i)\in Q$ converges to $(0,0)$, then
\begin{equation*}
  v_i:= \frac{p_i}{|p_i|}
\end{equation*}
converges (after passing to a 
subsequence) to 
a unit vector $v$ in $L$,
and thus
\begin{equation}\label{L-distance}
   \dist(p_i/|p_i|, L) \to 0.
\end{equation}
\end{theorem}

\begin{proof}
Let 
\[
  r_i:= 
  |(p_i,t_i)|
  =
\max\{|p_i|, \, |t_i|^{1/2}\}.
\]
Dilate the flow $M(\cdot)$
parabolically by 
\[
  (x,t) \mapsto
  (r_i^{-1}x, r_i^{-2}t)
\]
to get a flow $M_i(\cdot)$.
The flows $M_i(\cdot)$ converge
to the tangent flow 
\[
t\in \RR 
\mapsto
M'(t)
=
\begin{cases}
\{q: \dist(q,L)= 2^{1/2}|t|^{1/2}\}
&\text{if $t\le 0$}, \\
\emptyset &\text{if $t>0$}.
\end{cases}
\]
By passing to a subsequence, we can assume that the spacetime
points 
\[
   (r_i^{-1}p_i, r_i^{-2}t_i)
\]
converge to a limit $(p',t')$
such that
\begin{equation}
\label{triad}
\begin{gathered}
\max\{|p'|, |t'|^{1/2}\} = 1,
\\
p'\in M'(t'), 
\\
t'\ge 0.
\end{gathered}
\end{equation}
Since $t'\ge 0$ and since $M'(t')$
is nonempty, we see that $t'=0$
and that
\[
   M'(t')=M'(0)= L.
\]
Thus, by~\eqref{triad}, $p'$ is a unit
vector in $L$.
Since $p_i/r_i$ converges to a
unit vector, $v_i=p_i/|p_i|$ converges
to the same unit vector.
Consequently,
\[
 \dist(p_i/|p_i|, L) 
 \to
 \dist(v,L) = 0.
\]
\end{proof}

\begin{corollary}
\label{cylinder-corollary}
Suppose that $t_i\ge 0$ converges to $0$,
that $\mu_i$ tends to infinity, and that $q_i$ are points
with $\sup|q_i|<\infty$.
Then the surfaces
\begin{equation}\label{the-surfaces}
   \mu_i M(t_i) + q_i
   \tag{*}
\end{equation}
converge  to a subset of a line.
\end{corollary}

\begin{proof}
By passing to a subsequence, we 
can assume that the surfaces
$\mu_iM(t_i)$ converge 
to a set $S$, and that the $q_i$
converge to a point $q$.
By~\eqref{L-distance} in Theorem~\ref{cylinder-theorem}, 
the set $S$ is contained in the line $L$.
Thus the surfaces~\eqref{the-surfaces}
converge to
$S+q$, which is contained in the line $L+q$.
\end{proof}

\begin{theorem}
\label{morse-theorem}
Under the hypotheses of Theorem~\ref{cylinder-theorem},
there exist $R>0$ and $\eps>0$
with the following property.
For each $t\in [0,\eps]$, the
function
\[
 f_t:  p\in \Reg M(t)\cap B(0,R)
   \mapsto |p|
\]
is a smooth Morse function,
and each critical point is a local maximum or a local minimum.
\end{theorem}

\begin{proof}
Suppose not. Then there exist
regular points 
 $(p_i,t_i)\in Q$ converging to 
  $(0,0)$ such that $p_i$ is a critical point of $f_{t_i}$ and
  such that either (1) it has positive nullity, or (2) it has nullity $0$ and index $1$.

Let
\[
   h_i := h(M(t_i),p_i),
\]
be the length of the mean curvature vector of $M(t_i)$ 
at the point~$p_i$.
Translate $0$ and $M(t_i)$ by $-p_i$
and then dilate by $h_i$ to get
\begin{align*}
    q_i &:= h_i(0-p_i) = -h_ip_i ,\\
    M_i &:= h_i(M(t_i)-p_i).
\end{align*}

By 
Theorem~\ref{summary-theorem},  $h_i\to\infty$, and we can
assume (by passing to a subsequence)
that the $M_i$
converge smoothly to a limit $M'$
containing $0$, where
\begin{equation}
    \label{trichotomy}
\begin{gathered}
\text{$M'$ is a
compact, uniformly convex
surface,}
\\
\text{or a bowl soliton, or
a cylinder.}
\end{gathered}
\end{equation}

\begin{claim} 
$h_i|p_i|\to \infty$.
\end{claim}

\begin{proof}[Proof of claim]
Suppose not. Then we can assume that  $-h_ip_i$ converges to a finite limit $q$.
By Corollary~\ref{cylinder-corollary}, the surfaces
\[
 h_i(M(t_i)-p_i) = h_iM(t_i) - h_ip_i
\]
converge (after passing to a subsequence) to a subset of
the line $L+q$.
But that is impossible since
those surfaces converge to $M'$, 
which is not contained in a line.
Thus $h_i|p_i|\to\infty$, as claimed.
\end{proof}

Let
\[
   f_i(x) 
   =
   \dist(x, - h_ip_i) 
   -
   \dist(0, -h_ip_i).
\]
Since $h_i|p_i|\to\infty$
and since $p_i/|p_i|$ converges
to a unit vector $v\in L$,
we see that $f_i(\cdot)$
converges smoothly to the function
\begin{align*}
 &f: M'\to \RR, \\
 &f(x) = x\cdot v.
\end{align*}
Note also that $0$ is a critical point of $x$, and either it has
nullity $>0$, or it has nullity $0$ and index $1$.

Thus we see that either:
\begin{enumerate}
\item at least one of the principal
curvatures of $M'$ is $0$, or
\item one of the principal curvatures is positive and the
other is negative.
\end{enumerate}
Thus $M'$ must
be a cylinder by~\eqref{trichotomy}.
By~\cite{cm-new},  the axis of the cylinder
is parallel to $L$.
But then $f$ has no critical points since $v\in L$,
a contradiction.
\end{proof}

\begin{corollary}
\label{morse-corollary}
For every regular time
 $t\in (0,\eps]$ and for every
 $0<r\le R$
\[
   M(t)\cap B(0,r)
\]
is a disjoint union of
disks and spheres.
\end{corollary}

\begin{corollary}
\label{sc-corollary}
Suppose $t_i>T$ are regular times converging to $T$.
Suppose $\gamma_i$ are simple closed curves in $M(t_i)$ converging to $p$.
(That is $\gamma_i\subset B(p,r_i)$ for some $r_i\to 0$.)
Then for all sufficiently large $i$, $\gamma_i$ bounds a disk in $M(t_i)\cap B(p,r_i)$.
\end{corollary}

\section{Geodesics}
\label{geodesics-section}

\begin{theorem}\label{geodesics-theorem}
Suppose that $t\in [a,b]\mapsto M(t)$
is a $2$-dimensional 
standard Brakke flow 
in a smooth 
 Riemannian $3$-m manifold.
Suppose  that $T\in I$
is a time
such that $\Sing M(T)$ is compact
and 
  at which each tangent flow is given
by a multiplicity~one shrinker that is a plane, a sphere, or a cylinder.
Suppose $t_i>T$ are regular times that converge to $T\in I$.
Let $\Gamma_i$ be a $1$-cycle that minimizes length in its
mod~$2$ homology class in $M(t_i)$.
(Thus $\Gamma_i$ is a finite collection of disjoint, simple closed geodesics, each with multiplicity one.)
Then, as $i\to\infty$,  $\Gamma_i$ is bounded away
from $\Sing M(T)$.
That is, there is an $\eps>0$
such that, for all sufficiently large $i$, no point of $\Gamma_i$
is within distance $\eps$
of any point in $\Sing M(T)$.
\end{theorem}

\begin{proof}
Suppose not.  Then (after passing to a subsequence) there exist  $p_i\in \Gamma_i$ such that $\dist(p_i, \Sing M(T)) \to 0$.
Let $\gamma_i$ be the connected
component of $\Gamma_i$
that contains~$p_i$.
By passing to a further subsequence, we may assume that $p_i$ converges to a point $p\in \Sing M(T)$.
By Lemma~\ref{sphere-lemma}, the tangent flow to $M(\cdot)$ at $(p,T)$
must be a shrinking cylinder.

By Theorem~\ref{summary-theorem}, the
 length~$h_i$ of 
 the mean curvature vector of $M(t_i)$ at $p_i$ tends to $\infty$,
 and the surfaces 
\[
  M_i' := h_i(M_i-p_i)
\]
converge smoothly 
 (after passing to 
 a subsequence) to a limit $M'$
 that is one of the following:
\begin{enumerate}
\item a compact, uniformly convex  surface.
\item a bowl soliton.
\item a cylinder.
\end{enumerate}

It follows that the curves
\[
  \gamma_i':= h_i(\gamma_i-p_i)
\]
converge smoothly (after passing to a further subsequence) to a geodesic $\gamma'$ that is minimizing mod $2$ in $M'$.

Neither a compact convex set nor
a bowl soliton contains a geodesic
that is minimizing mod $2$. (See Lemma~\ref{bowl-lemma} below.)

Thus $M'$ is a cylinder $S\times \RR$, where $S$ is a circle.

Note that the lift of $\gamma'$ to the universal cover (a plane) of $M'$ is minimizing mod $2$,
and therefore is a single line. 
Thus $\gamma'$ is either
a closed curve $S\times \{c\}$ for some $c$,
or a geodesic of the form
\begin{equation}
\label{geodesic-form}
  \{ (\cos\theta, \sin \theta, v): \theta = sv + \alpha\}
\end{equation}
for some $s$.
(In fact, the minimizing mod $2$ property implies that $s=0$, but we don't need that fact.)

If $\gamma'$ were a closed curve, then $\gamma_i$ would 
be a simple closed curve converging to $p$.
But then,
by Corollary~
\ref{sc-corollary},
given $r>0$, $\gamma_i$
would (for all sufficiently 
  large~$i$)
  bound a disk
  in $M(t_i)\cap B(p,r)$,
 contrary to the minimizing property of $\Gamma_i$.
 
Thus $\gamma'$ is a nonclosed geodesic of the
 form~\eqref{geodesic-form}.  
Let $\delta$ be a closed geodesic in $M'$.   Then there exist simple closed curves $\delta_i$ in $M(t_i)$
such that
\[
    \delta_i':= h_i(\delta_i - p)
\]
converges smoothly to $\delta$.

Since $\delta$ intersects $\gamma'$ transversely in exactly one point, we see that, for large $i$,
$\delta_i$ intersects $\gamma_i$ transversely in exactly one point.
Thus $\delta_i$ is homologically nontrivial in $M(t_i)$.
But that contradicts
 Corollary~\ref{sc-corollary}, since $\delta_i$ converges to $p$.
\end{proof}

In the proof of
Theorem~\ref{geodesics-theorem},
we used the following lemma:

\begin{lemma}\label{bowl-lemma}
Let $M'$ be a bowl soliton, and let $\gamma'$ be a properly embedded geodesic in $M'$.
Then $M'$ is not homologically minimizing mod~$2$ in $M'$.
\end{lemma}

\begin{proof}
By translating, we can assume
that $M'$ intersects its axis of symmetry at the origin.
Note that $\gamma'$ cannot be 
a closed curve, since $M'$ is simply connected.
Let $q$ be a point of $\gamma'$
closest to $0$.

For large $R$, let $\gamma'_R$
be the connected component of 
 $\gamma'\cap B(0,R)$
 that contains $q$.
Then 
\[
 \length(\gamma'_R)
 \le
 \frac12 \length(M'\cap \partial B(0,R)).
\]
Dividing by $R$ and then letting
$R\to\infty$ gives
\[
  2 \le 0,
\]
a contradiction. (The right hand side is $0$ because $M'$ is asymptotically cylindrical.)
\end{proof}

\begin{remark}\label{geodesics-remark}
Theorem~\ref{geodesics-theorem} remains true
if the  minimizing
hypothesis on $\gamma_i$
is be replaced by the weaker hypotheses: 
$\Gamma_i$ has finite length,
and 
there is a neighborhood $U$ of 
 $\Sing M(T)$
such that for all sufficiently large~$i$, 
\[
    \Gamma_i\cap U
\]
is homologically minimizing 
  mod~$2$ in $M(t_i)\cap U$.
That is, 
\[
 \length(\Gamma_i)
 \le
 \length(\Gamma_i + \partial Q)
\]
for every $2$-chain $Q$(mod $2$)
that is compactly supported
in $M(t_i)\cap U$.
No changes are required in the
proof.
\end{remark}

\section{Genus}
\label{genus-section}

Let $M$ be a surface, possibly incomplete and
possibly non-orientable.
Note that $\genus(M)$ is the supremum of numbers $n$
with the following property: there exist $1$-cycles $\alpha_1,\dots,\alpha_n$
and $\beta_1, \dots, \beta_n$ in $M$ such that the mod $2$ intersection number
$I(\alpha_i, \beta_i)$
of $\alpha_i$ and $\beta_j$ is
given by
\[
I(\alpha_i,\beta_j)
= \delta_{ij} := 
\begin{cases}
1 &\text{if $i=j$, and} \\
0 &\text{if $i\ne j$}.
\end{cases}
\]

\begin{theorem}
\label{main-theorem}
Suppose that $t\in [a,b]\mapsto M(t)$
is a $2$-dimensional standard Brakke flow
in a smooth Riemannian $3$-manifold~$N$. Suppose that $M(a)$
is compact and that the Ricci curvature of $N$ is bounded below.
Suppose also that each tangent flow to $M(\cdot)$ has multiplicity~one,
and that the corresponding
shrinker is a plane, a sphere,
or a cylinder.
Then
\[
  g(t):=\genus(\Reg M(t))
\]
is a decreasing function
on $[a,b]$.
\end{theorem}

\begin{proof}
Suppose that  $t_k>T$ are regular 
times converging to~$T\in I$,
and that
\[
   g(t_k)\ge n
\]
for some positive integer~$n$
and for all~$k$.
By Theorem~\ref{general-theorem},
it suffices to show 
 that~$g(T)\ge n$.

Since $g(t_k)\ge n$,
 there
are $1$-cycles $\alpha^k_i$
and $\beta^k_i$ ($1\le i\le n$)
such that
\[
   I(\alpha^k_i, \beta^k_j)
   = \delta_{ij}.
\]
We may choose each $\alpha_i^k$ and each $\beta_i^k$ to minimize
length in its mod $2$ homology class in~$M(t_k)$.

By Theorem~\ref{ricci-theorem},
the $M(t)$ all lie in a compact subset of $N$, and, therefore, so do
the $\alpha^k_i$ and $\beta^k_i$.

Note that if $U$ is a simply connected region in $M(t_k)$, then
\begin{equation}
\label{length-bounds}
\begin{aligned}
\length(\alpha^k_i\cap U)
 &\le 
\frac12 \length(\partial U), 
   \\
\length(\beta^k_i\cap U)
 &\le 
\frac12 \length(\partial U).
\end{aligned}
\end{equation}

By Theorem~\ref{geodesics-theorem}
and the 
length bounds~\eqref{length-bounds}, we can assume, after passing to a subsequence, that, for each $i\in \{1,2,\dots,n\}$,
$\alpha^k_i$ and $\beta^k_i$ converge smoothly as $k\to\infty$
to geodesic $1$-cycles $\alpha_i$ and $\beta_i$ in $\Reg M(T)$.
By smooth convergence, 
\[
I(\alpha_i,\beta_j)=\delta_{ij}
\]
for $1\le i,j\le n$.
Hence 
\[
  g(T)=\genus(\Reg M(T))\ge n.
\]
\end{proof}

\section{A Localized Version}
\label{local-section}

\begin{theorem}
\label{local-open}
Suppose $t\in [a,b]\mapsto M(t)$
is a standard Brakke
flow in a smooth
Riemannian $3$-manifold~$N$.

Suppose $W$ is an open subset
of~$N$ such that 
\begin{enumerate}
\item
\label{one}
$\partial W$ is smooth.
\item
\label{two}
$\partial W$ is disjoint
from $\Sing M(t)$ for 
 each $t\in I$.
\item
\label{three}
$\Reg M(t)$ is transverse
 to $\partial W$ for
 each $t\in I$.
\item
\label{four}
$\cup_{t\in I}(W\cap M(t))$
is contained in 
a compact subset of~$N$.
\item
\label{five}
Each tangent flow
to~$M(\cdot)$
at each spacetime point in $W\times I$
has multiplicity~one, and
 the corresponding
shrinker is a plane, a sphere,
or a cylinder.
\end{enumerate}
Then 
\begin{equation}
\label{W-conclusion}
 g(t):= \genus(W\cap \Reg M(t))
 \tag{*}
\end{equation}
is a decreasing function
 of $t\in I$.
\end{theorem}

\begin{remark}
By Theorem~\ref{ricci-theorem},
Hypothesis~[4] is satisfied if
 $M(a)$ is compact
and if $N$ is a complete manifold with Ricci curvature bounded below.
\end{remark}

\begin{proof}
Note that
\begin{equation}
\label{equal}
\genus(W\cap \Reg M(t))
=
\genus(\overline{W}\cap \Reg M(t))
\end{equation}
for all $t\in [a,b]$.
Consider the Brakke flow
\[
 t\in [a,b] \mapsto \tilde M(t)
  = M(t)\cap W.
\]
Suppose $t_k>T$ are regular times of $\tilde M(\cdot)$
converging to $T\in I$,
and that
\[
   g(t_k)\ge n
\]
for all $k$ (where $n$ is a positive integer).
By Theorem~\ref{general-theorem}, it suffices to show
that $g(T)\ge n$.

Since $g(t_k)\ge n$, there
exist $1$-cycles $\alpha^k_i$
and each $\beta^k_i$
 ($1\le i\le n$) in 
 $M(t_i)\cap \overline{W}$
such that
\[
  I(\alpha^k_i, \beta^k_j)
  =
  \delta_{ij}.
\]
We may choose each $\alpha^k_i$ and $\beta^k_j$
to minimize 
length in its mod~$2$ homology
class.
Note that each of the cycles
is a disjoint union of 
simple closed
 curves.  The curves are 
  $C^{1,1}$,  and they are smooth geodesics away from $\partial W$.

By Theorem~\ref{geodesics-theorem} and
Remark~\ref{geodesics-remark}, 
the $\alpha^k_i$ and 
  $\beta^k_i$
are bounded away from 
  $\sing M(T)$
  as $k\to\infty$.

Thus, as in the proof of 
 Theorem~\ref{main-theorem},  after passing to a  subsequence, the
 $\alpha^k_i$ and
 $\beta^k_i$ converge 
 (as $k\to\infty$)
to cycles $\alpha_i$ and $\beta_i$
in $\overline{W}\cap \Reg M(T)$.
Note that the convergence is
in $C^1$ 
(and, indeed, is smooth away from $\partial W$.)
Hence
\[
  I(\alpha_i, \beta_j)
  =
  \delta_{ij}
\]
for $1\le i,j\le n$.
Therefore, $\overline{W}\cap \Reg M(T)$ has genus $\ge n$.
Hence $g(T)\ge n$ 
 by~\eqref{equal}.
\end{proof}

\begin{remark}
\label{local-open-remark}
More generally, one can allow open sets that depend on time.  Let $t\in I\mapsto W(t)$ be 
a continuous, one-parameter family of open
sets such that $\partial W(t)$ is a smooth hypersurface that depends smoothly on~$t$.
Then Theorem~\ref{local-open} remains true if we replace $W$ by $W(t)$  in~\eqref{one}--\eqref{four}
and~\eqref{W-conclusion}, 
and if we replace
 $W\times I$
by $\{(p,t): t\in I,\, p\in W(t)\}$ in~\eqref{five}.
The proof is exactly the same.
\end{remark}

\section{Another Localized Version}
\label{another-local-section}

\begin{theorem}
\label{local-curves}
Suppose that $t\in [a,b]\mapsto M(t)$
is a two-dimensional, 
standard Brakke flow
in a complete, smooth Riemannian
$3$-manifold with Ricci curvature
bounded below.
Suppose that $M(a)$ is compact. 
Suppose that each tangent flow to $M(\cdot)$ has multiplicity~one,
and that the corresponding
shrinker is a plane, a sphere,
or a cylinder.
Suppose that $\Gamma$ is a finite, disjoint union of circles,
and that 
\[
  f: I\times \Gamma \to N
\]
is a smooth map such that
 for each $t\in I$, $f(t,\cdot)$ is a smooth immersion
 of $\Gamma$ into $\Reg M(t)$ with only transverse intersections.
For $t\in I$, let 
 $K_t:=f(t,\Gamma)$, and suppose that the $K_t$'s are
 all homeomorphic.
Then
\[
 g(t):=\genus(\Reg M(t)\setminus K_t)
\]
is a decreasing function 
of $t\in [a,b]$.
\end{theorem}

\begin{proof}
For $t\in [a,b]$,
let $\Sigma(t)$ be the geodesic
completion of 
  $M(t)\setminus K_t$,
and let 
\[
  \tilde M(t)
  :=
  \Sigma(t) \setminus 
  \Sing M(t).
\]
Note that
\begin{equation}
\label{tilde}
  g(t) = \genus(\tilde M(t)).
\end{equation}
Suppose that $t_k>T$ are regular times converging to $T$ and that 
\[
g(t_k)\ge n
\]
for all $k$ (where $n$ is a positive integer).  
By Theorem~\ref{general-theorem}, 
it suffices to prove 
that $g(T)\ge n$.

By~\eqref{tilde}
and by the assumption 
that $g(t_k)\ge n$,
there are $1$-cycles 
 $\alpha^k_i$
 and
 $\beta^k_i$
 ($1\le i \le n$)
 in $\tilde M(t_k)$
 such that
 \[
 I(\alpha^k_k,\beta^k_i)
 =
 \delta_{ij}.
 \]
 We can choose each
 $\alpha^k_i$ and each $\beta^k_i$
 to minimize length
 in its mod $2$ homology class in $\tilde M(t_k)$.
 Thus each such $1$-cycle
 is a finite union of disjoint, $C^{1,1}$ curves,
 and the curves are geodesics
 away from 
    $\partial \tilde M(t_k)$.

For large $k$, the $1$-cycles
are uniformly bounded away
from $\Sing M(T)$ by 
 Theorem~\ref{geodesics-theorem} and
 Remark~\ref{geodesics-remark}.
By Theorem~\ref{ricci-theorem},
the $1$-cycles are all contained
in compact subset of $N$.

 Note also that the geodesic
 curvatures of the $\alpha^k_i$ and
  $\beta^k_i$ 
 are uniformly bounded
 (since the curvatures of the 
 curves in $K_t$ are uniformly bounded.)

 Thus, after passing
 to a subsequence, these $1$-cycles converge in $C^1$
 (indeed, in $C^{1,\eta}$ for 
 every $\eta\in (0,1)$)
 to limit cycles 
  $\alpha_i$ and $\beta_i$
 that are compactly
 supported in
 $\tilde M(T)\setminus\Sing M(T)$.
 
 Let $C$ be the set of corner points of $\partial \tilde M(T)$, i.e., the set
 of points in $\partial \tilde M(T)$ corresponding
 to the self-intersections 
 of $K_T$.
 Since the $\alpha_i$ and $\beta_i$ are $C^1$, they do not contain any points of $C$.
 Note that the $\tilde M(t_i)$
 converge smoothly to 
   $\tilde M(t_i)$
 away from $C\cup \Sing M(T)$.
 Thus
 \[
   I(\alpha_i, \beta_j)=\delta_{ij}
 \]
for $1\le i, j \le n$.
 Hence $g(T)\ge n$
 by~\eqref{tilde}.
\end{proof}

\begin{corollary}
Suppose that $t\in [a,b]\mapsto M(t)$
is a $2$-dimensional 
standard Brakke flow in a
complete, smooth Riemannian $3$-manifold~$N$
with Ricci curvature bounded below.
Suppose $M(a)$ is compact.
  Suppose that all
tangent flows have multiplicity one and are given by planes, spheres, or cylinders.
Suppose that also that $K$ is the union of a finite collection of smooth, simple closed curves that intersect each other transversely, and that $K$
is contained in $\Reg M(t)$
for each $t\in [a,b]$.
Then
\[
 g(t):=
 \genus(\Reg M(t)\setminus K)
\]
is a decreasing function of $t\in [a,b]$.
\end{corollary}

\section{Appendix}

Let  $d_1$, and $d_2$
be the entropy of the cylinder and sphere, respectively.
Recall that 
\[
  2 > d_1 > d_2 > 1, 
\]
and that there is a $\delta>d_1$ such that 
 the plane, the sphere, and the cylinder are the only smooth shrinkers in $\RR^3$ with entropy $<\delta$
   \cite{bw}*{Corollary~1.2}.
By replacing $\delta$ by a smaller number $> d_1$, we can assume that $\delta< 2$.

\begin{theorem}
\label{appendix-theorem}
Let $M$ be a compact, smoothly embedded surface in a complete, smooth Riemannian $3$-manifold with Ricci curvature bounded below.
Let 
\[
  t\in [0,\infty)\mapsto \mu(t)
\]
be a unit-regular, cyclic Brakke flow with 
\[
  \mu(0)=\Hh^2\llcorner M.
\]
(Such a flow exists by elliptic regularization.)
Let $T^*$ be the infimum of times $t>0$ at which there is a point at which the Gauss density is $\ge \delta$, i.e.,
at which a shrinker has entropy $\ge \delta$.
Then
\begin{enumerate}
\item
\label{app1}
If $T^*<\infty$, then there is a spacetime point $(p,T^*)$ at which a shrinker
has entropy $\ge \delta$.
\item
\label{app2}
Each shrinker at a time $<T^*$ is a plane, a sphere, or a cylinder, and has multiplicity $1$.
\item 
\label{app3}
$\mu(t)=\Hh^2(t)\llcorner M(t)$
  for $t<\Tpos$.
\item
\label{app4}
The flow
\[
   t\in [0,\Tpos)\mapsto \mu(t)
\]
is standard.
\end{enumerate}
\end{theorem}

\begin{proof}
Assertion~\eqref{app1} follows trivially from the upper semicontinuity of Gauss density.

For any $2$-dimensional cyclic flow in a $3$-manifold, any shrinker at a spacetime point of Gauss density $<2$ is smoothly embedded and has multiplicity one. 
(See~\cite{white-boundary-sing}*{Lemma~20}.) At times $<T^*$, the Gauss density is 
  $\le \delta$, the shrinker 
  must be a plane, sphere,
  or cylinder and must have multiplicity one.
  Thus Assertion~\eqref{app2}
holds.

Assertion~\eqref{app3} follows from the uniqueness theorem in~\cite{chhw}.

Assertion~\eqref{app4} holds 
by Assertion~\eqref{app3} and the assumption that $\mu(\cdot)$ is a standard Brakke flow.
\end{proof}

\begin{corollary}
\label{T*-Tpos-corollary}
The $T^*$ in Theorem 21 is equal to the $\Tpos$ in Theorem~\ref{intro-theorem}.
\end{corollary}

So far in this paper, we have not invoked
the powerful recent results of Bamler and Kleiner.
That was necessary to avoid
circularity if we regard this paper as providing a (detailed) proof of \cite{bk}*{Lemma~7.8(e)}, as mentioned in the introduction.
Now, having proved that Lemma, we can invoke \cite{bk} without circularity.
In particular, we have:

\begin{theorem}
\label{Tpos-shrinker-theorem}
In Theorem~\ref{intro-theorem},
if $\Tpos<\infty$, then 
every  shrinker that occurs at time $\Tpos$ is a smooth embedded surface of  of multiplicity one, and
at least one such shrinker has genus  $\ge 1$.
\end{theorem}

\begin{proof}
By Theorem~\ref{intro-theorem}
and~\cite{bk}, all shrinkers
at time~$\Tpos$ are smoothly embedded and have multiplicity~$1$. By Theorem~\ref{appendix-theorem}(1), at least one such shrinker has entropy
 $\ge \delta$ and thus is not a plane, sphere, or cylinder.
Thus, by~\cite{brendle}, it
 has positive genus.
\end{proof}


\begin{bibdiv}
\begin{biblist}

\bib{brendle}{article}{
   author={Brendle, Simon},
   title={Embedded self-similar shrinkers of genus 0},
   journal={Ann. of Math. (2)},
   volume={183},
   date={2016},
   number={2},
   pages={715--728},
   issn={0003-486X},
   review={\MR{3450486}},
   doi={10.4007/annals.2016.183.2.6},
}


\bib{bk}{article}{
      title={On the Multiplicity One Conjecture for Mean Curvature Flows of surfaces}, 
      author={Bamler, Richard},
      author={Kleiner, Bruce},
      year={2023},
      pages={1--58},
      eprint={arxiv:2312.02106},
      doi={10.48550/arXiv.2312.02106},
}

\bib{bw}{article}{
   author={Bernstein, Jacob},
   author={Wang, Lu},
   title={A topological property of asymptotically conical self-shrinkers of
   small entropy},
   journal={Duke Math. J.},
   volume={166},
   date={2017},
   number={3},
   pages={403--435},
   issn={0012-7094},
   review={\MR{3606722}},
   doi={10.1215/00127094-3715082},
}


\bib{ccms-generic}{article}{
   author={Chodosh, Otis},
   author={Choi, Kyeongsu},
   author={Mantoulidis, Christos},
   author={Schulze, Felix},
   title={Mean curvature flow with generic initial data},
   journal={Invent. Math.},
   volume={237},
   date={2024},
   number={1},
   pages={121--220},
   issn={0020-9910},
   review={\MR{4756990}},
   doi={10.1007/s00222-024-01258-0},
}








\bib{chhw}{article}{
   author={Choi, Kyeongsu},
   author={Haslhofer, Robert},
   author={Hershkovits, Or},
   author={White, Brian},
   title={Ancient asymptotically cylindrical flows and applications},
   journal={Invent. Math.},
   volume={229},
   date={2022},
   number={1},
   pages={139--241},
   issn={0020-9910},
   review={\MR{4438354}},
   doi={10.1007/s00222-022-01103-2},
}


\bib{cm-new}{misc}{
      author={Colding, Tobias~Holck},
      author={Minicozzi, William~P., II},
       title={Quantitative uniqueness for mean curvature flow},
        date={2025},
      eprint={2502.03634},
        note={arXiv:2502.03634 [math.DG]},
         url={https://arxiv.org/abs/2502.03634},
}


\bib{cm-unique}{article}{
   author={Colding, Tobias Holck},
   author={Minicozzi, William P., II},
   title={Uniqueness of blowups and \L ojasiewicz inequalities},
   journal={Ann. of Math. (2)},
   volume={182},
   date={2015},
   number={1},
   pages={221--285},
   issn={0003-486X},
   review={\MR{3374960}},
   doi={10.4007/annals.2015.182.1.5},
}





\bib{hmw-shrinkers}{article}{
  author={Hoffman, David},
  author={Martin, Francisco},
  author={White, Brian},
  title={Generating Shrinkers by Mean Curvature Flow},
  date={2025},
  eprint={https://arxiv.org/abs/2502.20340}
}

\bib{hw-avoid}{article}{
   author={Hershkovits, Or},
   author={White, Brian},
   title={Avoidance for set-theoretic solutions of mean-curvature-type
   flows},
   journal={Comm. Anal. Geom.},
   volume={31},
   date={2023},
   number={1},
   pages={31--67},
   issn={1019-8385},
   review={\MR{4652509}},
   doi={10.4310/cag.2023.v31.n1.a2},
}

\bib{white-95}{article}{
   author={White, Brian},
   title={The topology of hypersurfaces moving by mean curvature},
   journal={Comm. Anal. Geom.},
   volume={3},
   date={1995},
   number={1-2},
   pages={317--333},
   issn={1019-8385},
   review={\MR{1362655}},
   doi={10.4310/CAG.1995.v3.n2.a5},
}

\bib{white-stratification}{article}{
   author={White, Brian},
   title={Stratification of minimal surfaces, mean curvature flows, and
   harmonic maps},
   journal={J. Reine Angew. Math.},
   volume={488},
   date={1997},
   pages={1--35},
   issn={0075-4102},
   review={\MR{1465365}},
   doi={10.1515/crll.1997.488.1},
}

\bib{white-cyclic}{article}{
   author={White, Brian},
   title={Currents and flat chains associated to varifolds, with an
   application to mean curvature flow},
   journal={Duke Math. J.},
   volume={148},
   date={2009},
   number={1},
   pages={41--62},
   issn={0012-7094},
   review={\MR{2515099}},
   doi={10.1215/00127094-2009-019},
}


\bib{white-topological}{article}{  
   author={White, Brian},
   title={Topological change in mean convex mean curvature flow},
   journal={Invent. Math.},
   volume={191},
   date={2013},
   number={3},
   pages={501--525},
   issn={0020-9910},
   review={\MR{3020169}},
   doi={10.1007/s00222-012-0397-0},
}

\bib{white-mcf-boundary}{article}{
   author={White, Brian},
   title={Mean curvature flow with boundary},
   journal={Ars Inven. Anal.},
   date={2021},
   pages={Paper No. 4, 43},
   review={\MR{4462472}},
}

\bib{white-boundary-sing}{article}{
   author={White, Brian},
   title={Boundary singularities in mean curvature flow and total curvature
   of minimal surface boundaries},
   journal={Comment. Math. Helv.},
   volume={97},
   date={2022},
   number={4},
   pages={669--689},
   issn={0010-2571},
   review={\MR{4527825}},
   doi={10.4171/cmh/542},
}

\bib{genus-II}{article}{
  author={White, Brian},
  title={The genus-decreasing property of mean curvature flow, II},
  date={2026},
  note={In preparation},
}



\end{biblist}

\end{bibdiv}
\end{document}